\documentclass[11pt,oneside,reqno]{amsart}
\usepackage{enumerate}
\usepackage{mathrsfs}
\usepackage{eucal}
\usepackage{geometry}
\usepackage{romanbar}
\usepackage[colorlinks=true,allcolors=blue]{hyperref}

\numberwithin{equation}{section}

\theoremstyle{definition}
\newtheorem{Definition}{Definition}[section]
\newtheorem{Example}[Definition]{Example}
\newtheorem{Remark}[Definition]{Remark}
\theoremstyle{plain}
\newtheorem{Theorem}[Definition]{Theorem}
\newtheorem{Proposition}[Definition]{Proposition}
\newtheorem{Corollary}[Definition]{Corollary}
\newtheorem{Lemma}[Definition]{Lemma}

\newcommand{\al}{\alpha}
\newcommand{\be}{\beta}
\newcommand{\ep}{\varepsilon}

\newcommand{\Z}{\mathbb{Z}}

\newcommand{\C}{\mathbb{C}}

\newcommand{\Fsl}{\mathfrak{sl}}
\newcommand{\Fsp}{\mathfrak{sp}}
\newcommand{\Fso}{\mathfrak{so}}
\newcommand{\Fg}{\mathfrak{g}}
\newcommand{\Fh}{\mathfrak{h}}
\newcommand{\Fn}{\mathfrak{n}}

\DeclareMathOperator{\ad}{ad}
\DeclareMathOperator{\diag}{diag}

\newcommand{\p}{\partial}

\renewcommand{\tilde}{\widetilde}
\newcommand{\II}{\textup{\Romanbar{II}}}

\title[Exact Solutions to the Klein--Gordon Equation via Reduction Algebras]{Exact Solutions to the Klein--Gordon Equation\\via Reduction Algebras}
\author{Jonas T. Hartwig \and Lillian Ryan Uhl \and Dwight Anderson Williams II}
\date{}

\address{Department of Mathematics, Iowa State University, Ames, IA-50011, USA}
\email{jth@iastate.edu}
\urladdr{http://jthartwig.net}
\thanks{J.T.H. and L.R.U. were supported by the Army Research Office grant W911NF-24-1-0058. D.A.W. II was supported by the National Science Foundation under Award No.~2503008.}
\address{Department of Mathematics, Iowa State University, Ames, IA-50011, USA}
\email{lillianu@iastate.edu}
\urladdr{}
\address{Department of Mathematics, Morgan State University, Baltimore, MD-21251, US}
\email{dwight@mathdwight.com}
\urladdr{https://mathdwight.com}

\begin{document}

\begin{abstract}
Reduction algebras (also known as generalized Mickelsson algebras, Zhelobenko algebras, or transvector algebras) are well-studied associative algebras appearing in the representation theory of Lie algebras. In the 1990s, Zhelobenko noted that reduction algebras have a connection to field equations from physics, whereas Howe's study of dual pairs in the 1980s signifies that the link originated even earlier. In this paper, we revisit the simplest case of a massless scalar field, in arbitrary spacetime dimension $n\ge 3$ and arbitrary flat metric $\eta_{ab}$; the field equation then specializes to the Laplace equation and to the massless Klein--Gordon equation. The relevant reduction algebra comes from a representation of $\Fsl_2$ (technically, $\Fsp_2$) by differential operators. We give a complete presentation of this reduction algebra. Its raising operators produce bosonic states $|a_1a_2\cdots a_\ell\rangle$, polynomial solutions of degree $\ell$ to the field equation, and we show that the natural bilinear form on these states is computed by symmetrizers built from the rational dynamical R-matrix. Along the way we give explicit formulas for the states, and a self-contained proof, for a general metric, of Zhelobenko's results that they span the polynomial solutions and that the solution space is an irreducible module over the reduction algebra.
\end{abstract}

\maketitle

\tableofcontents

\section{Introduction}

The Klein--Gordon equation $\square \phi = m^2 \phi$ is the equation of motion for linear scalar field theory, where $\phi$ is a scalar field of mass $m$ and $\square$ is the d'Alembert operator. Thus the Klein--Gordon equation plays a fundamental role in physics with a history \cite{kemmerParticleAspectMeson1939} of mathematical interest, especially in finding numerical solutions as done in \cite{shakeriNumericalSolutionKlein2007}. In this paper we consider the algebraic aspects of the massless Klein--Gordon equation, specifically the connections to reduction algebras \cite{Mic1973,van1976harisch,Zhe1989,Zhe1994,KhoOgi2008,khoroshkinDiagonalReductionAlgebras2010,debieHarmonicTransvectorAlgebra2017,hartwigDiagonalReductionAlgebra2022c,eelbodeOrthogonalBranchingProblem2022,hartwigGhostCenterRepresentations2023,eelbodeHermitianRefinementSymplectic2024,hartwigSymplecticDifferentialReduction2025,mudrovMickelssonAlgebrasHasse2025} and dynamical quantum groups \cite{Fel1995,etingof1999exchange,arnaudon1997universal,Kho2004,kalmykovGeometricCategoricalApproaches2021,kalmykovCategoricalApproachDynamical2022}. More precisely, we study structures on the space $\mathscr{F}^+$ of all polynomials $\phi=\phi(x^1,x^2,\ldots,x^n)$ in $n$ commuting variables $x^i$ that satisfy the equation
\begin{equation}\label{eq:KG}
\sum_{a,b=1}^n \eta^{ab} \partial_a\partial_b \phi = 0,
\end{equation}
where $\partial_a=\frac{\partial}{\partial x^a}$, and $\eta^{ab}$ are the entries of (the inverse of) an arbitrary non-degenerate symmetric real matrix $(\eta_{ab})$ called the flat metric. Two interesting and commonly studied cases are when the metric is the identity matrix (equivalently, when the $\eta^{ab}$ are given by the Kronecker delta $\delta^{ab}$) and when the metric is the diagonal matrix $\diag(1,-1,-1,\ldots,-1)$. 
The equation \eqref{eq:KG} then becomes the Laplace equation $\Delta \phi=0$ and Klein--Gordon equation $\square \phi=0$, respectively.

It is well-documented \cite[for example]{binegarUnitarizationSingularRepresentation1991} that the operator $\partial_a\partial^a:=\sum_{a,b=1}^n \eta^{ab}\partial_a\partial_b$ is invariant under the special orthogonal group preserving the metric. The infinitesimal action, that is, the action of the corresponding Lie algebra\footnote{All algebras within are taken to be defined over the complex numbers.} may be expressed as the adjoint action of the differential operators $b_{ij}=x_i \partial_j - x_j\partial_i$ where $x_i=\sum_j\eta_{ij}x^j$. What is perhaps less known is that the operators $b_{ij}$ generate a subalgebra of a much larger algebra $Z^\circ$ that acts \emph{irreducibly} (when $n\ge 3$) on the solution space $\mathscr{F}^+$ to the equation $\partial_a\partial^a\phi=0$. In particular, the algebra $Z^\circ$ contains ``raising operators'' $\tilde x^i$ which take a degree $d$ solution $\phi$ of \eqref{eq:KG} to a degree $d+1$ solution $\tilde x^i\cdot\phi$, and every solution is obtained from the constant solution $1$ by applying raising operators and taking linear combinations. In the case of the Laplace operator, these facts are due to Zhelobenko: the raising operators appear in \cite[\S0.5]{Zhe1994}, where it is shown that every harmonic polynomial is obtained from $1$ by applying them, and the irreducibility is \cite[7.3.6]{Zhe1994} (see also \cite{Zhe1996}). Over $\C$, every non-degenerate metric is equivalent to the Euclidean one by a linear change of variables, so these results hold for any metric. Keeping $\eta_{ab}$ as a formal parameter is nevertheless convenient, since all formulas then apply verbatim to the Klein--Gordon case and to any real form. The Laplace operator and its relation to the orthogonal group is related to the Howe dual pair $(\Fsp_2,\Fso_n)$ of Lie algebras inside the symplectic Lie algebra $\Fsp_{2n}$; see \cite{How1985}.

The algebra $Z^\circ$ is a subalgebra of the reduction algebra $Z=Z(W'_\eta(2n),\Fsp_2)$ where $W'_\eta(2n)=\C(H)\otimes_{\C[H]} W_\eta(2n)$ and $W_\eta(2n)$ is the $n$:th Weyl algebra extended by central generators $\eta_{ab}$ (see Section \ref{sec:presentation} for details). Our main results are the following.
\begin{enumerate}[(1)]
\item We give a complete presentation of the reduction algebra $Z$ by generators and relations (Proposition \ref{prop:KG-relations}), and a basis when $\eta_{ab}$ is diagonal (Corollary \ref{cor:basis}).
\item The symplectic Fourier transform on the Weyl algebra induces an involution on $Z$, which defines a bilinear form $\langle\cdot|\cdot\rangle$ on the states $|a_1a_2\cdots a_\ell\rangle=\tilde x^{a_1}\cdots\tilde x^{a_\ell}\cdot 1$. We show that this form is computed by symmetrizers built from the rational dynamical $R$-matrix that appears when the relations of $Z$ are written in compact form (Theorem \ref{thm:inner-product}). This gives a direct link between the massless Klein--Gordon equation and integrable face models.
\end{enumerate}
Along the way, we give an explicit formula for the states $|a_1a_2\cdots a_\ell\rangle$ (Theorem \ref{thm:explicit-solutions}); they are the harmonic projections of the corresponding monomials (Remark \ref{rem:harmonic-projection}). We also give a self-contained proof, for a general metric, that these states span $\mathscr{F}^+$ and that $\mathscr{F}^+$ is irreducible over $Z^\circ$ (Section \ref{sec:irreducibility}).


Merging aspects of \cite{How1985}\cite{Zhe1994,Zhe1996}, we briefly review the connection between the Klein--Gordon equation and reduction algebras in more detail. This story can be generalized to many other equations and algebras. It begins with the observation that the Weyl algebra $W(2n)$ on $2n$ generators $x^a$, $\partial_a$ ($a=1,2,\ldots,n$) can be defined as the tensor algebra on $\C^{2n}$ modulo an ideal defined in terms of a non-degenerate anti-symmetric bilinear form. Thus $W(2n)$ carries a canonical action of the rank $n$ symplectic group by algebra automorphisms. When differentiated, this provides an action of the Lie algebra $\Fsp_{2n}$ on $W(2n)$ by derivations. Since any derivation of the Weyl algebra is inner, this gives a Lie algebra homomorphism from $\Fsp_{2n}$ to the Weyl algebra viewed as a Lie algebra with the commutator. This is known as the \emph{oscillator representation} of $\Fsp_{2n}$, see \cite{havlicekCanonicalRealizationsLie1976}\cite[Lemma 4.6.9(iv)]{dixmierEnvelopingAlgebras1996}\cite[§1.59]{FraSciSor2000}. Passing to the universal enveloping algebra we get the vertical arrow in the following diagram of associative algebras and algebra homomorphisms:
\begin{equation}\label{eq:diagram}
\begin{alignedat}{2}
U(\Fsp_2)\to U(&\Fsp_{2n})\leftarrow U(\Fso_n) \\ 
&\downarrow  \\ 
W&(2n)
\end{alignedat}
\end{equation}
The Lie algebra $\Fsp_{2n}$ contains a Howe dual pair $(\Fsp_2,\Fso_n)$. That is, there are Lie subalgebras $\Fsp_2\subset\Fsp_{2n}$ and $\Fso_n\subset\Fsp_{2n}$ which are mutual centralizers. These provide the horizontal maps in the diagram \eqref{eq:diagram}. Explicitly, the left one sends $\left(\begin{smallmatrix}a&b\\c&-a\end{smallmatrix}\right)\in\Fsp_2=\Fsl_2$ to $\left(\begin{smallmatrix}a\mathbf{1}_n&b\mathbf{1}_n\\c\mathbf{1}_n&-a\mathbf{1}_n\end{smallmatrix}\right)\in\Fsp_{2n}$ and the right map sends $x=-x^T\in\Fso_n$ to $\left(\begin{smallmatrix}x&0\\0&x\end{smallmatrix}\right)\in\Fsp_{2n}$. The composition $U(\Fsp_2)\to U(\Fsp_{2n})\to W(2n)$ can be slightly modified so that the positive simple root vector $e$ of $\Fsp_2$ is sent to the operator $\p_a\p^a$. 

Consequently, the space $\mathscr{F}$ polynomials in $n$ variables with complex coefficients, which is naturally a module over the Weyl algebra, can also be viewed as a module over $U(\Fsp_2)$. Furthermore, the space $\mathscr{F}^+$ of all solutions $\phi$ to the equation $\p_a\p^a\phi=0$ becomes identified with the subspace of primitive (that is, highest weight) vectors with respect to $\Fsp_2$. The normalizer $N$ of the left ideal $I=W(2n)\p_a\p^a$ in $W(2n)$ acts on the solutions space, and so does $N/I$. A localization $Z$ of $N/I$ is called the \emph{reduction algebra} associated to the algebras $U(\Fsp_2)$, $W(2n)$, and the map $U(\Fsp_2)\to W(2n)$. The algebra $Z^\circ$ is a subalgebra of $Z$ that has well-defined action on $\mathscr{F}^+$ (see Section \ref{sec:irreducibility} for the precise definition).

\section{Background}

In this section we review the details of the procedure by which one obtains a representation of the Lie algebra $\Fsl_2$ (technically, $\Fsp_2$) by differential operators which maps one of the nilpotent elements to the generalized Laplace-d'Alembert operator $\square=\p_a\p^a=\sum_{a,b}\eta^{ab}\p_a\p_b$ where $\eta^{ab}$ can be any constant non-degenerate symmetric matrix. 
We then recall the extremal projector for $\Fsl_2$ and its application to reduction algebras.

\subsection{Oscillator Representations}
The oscillator representation \cite[Ch. 4, §6 (Lemma 4.6.9(iv))]{dixmierEnvelopingAlgebras1996}\cite[§1.59]{FraSciSor2000} of the symplectic Lie algebra $\Fsp_{2m}$ is a homomorphism from the enveloping algebra $U(\Fsp_{2m})$ to the Weyl algebra $W(2m)$ on $2m$ generators. Since the $3$-dimensional symplectic Lie algebra $\Fsp_2$ is isomorphic to $\Fsl_2$, we use the latter notation since it is more commonly seen.

Let $\Fsl_2$ be the 3-dimensional complex simple Lie algebra of traceless $2\times 2$-matrices. It has a $\C$-vector space basis
\begin{equation}
e=\left(\begin{smallmatrix}0&1\\0&0\end{smallmatrix}\right),\quad
f=\left(\begin{smallmatrix}0&0\\1&0\end{smallmatrix}\right),\quad
h=\left(\begin{smallmatrix}1&0\\0&-1\end{smallmatrix}\right),
\end{equation}
which satisfy
\begin{equation}\label{eq:sl2-rels}
[h,e]=2e,\quad [h,f]=-2f,\quad [e,f]=h,
\end{equation}
where $[a,b]=ab-ba$ denotes the commutator. Let $W(2)$ be the first Weyl algebra, defined as the associative unital $\C$-algebra with two generators $\p$ and $x$ subject to the single relation $\p x - x\p = 1$. The \emph{oscillator representation} of $\Fsl_2$ (when viewed as $\Fsp_2$) is the homomorphism of associative unital $\C$-algebras 
\[
\varphi_1:U(\Fsl_2)\to W(2)
\]
given by
\begin{equation}\label{eq:oscillator-sp2}
\varphi_1(e)=\frac{i}{2}\p^2,\quad \varphi_1(f)=\frac{i}{2}x^2,\quad \varphi_1(h)=-\frac{1}{2}(x\p+\p x).
\end{equation}
We have chosen normalization constants for $e$ and $f$ so that matrix transposition (a $\C$-linear anti-automorphism of $\Fsl_2$, extended to $U(\Fsl_2)$) is represented by the symplectic Fourier transform $\ast$ on the Weyl algebra side. By definition, $\ast:W(2)\to W(2)$, $a\mapsto a^\ast$, is $\C$-linear, $(ab)^\ast=b^\ast a^\ast$, and $x^\ast=\p$, $\p^\ast=x$. 
Note also that in \eqref{eq:oscillator-sp2} we have flipped the role of $e$ and $f$ compared to the conventions in \cite{FraSciSor2000}. The reason is that we will be using the extremal projector for $\Fsl_2$ which traditionally projects onto highest rather than lowest weight spaces. Therefore we want the positive root vector $e$ to be mapped to the relevant partial differential operator. Here, $\p^2$ is simply the $1$-dimensional Laplace operator.

To get to higher dimensional Laplace operators, one can use the iterated comultiplication in the Hopf algebra $U(\Fsl_2)$. For this, fix a positive integer $n$, which will be the dimension of spacetime. The diagonal map $\Fsl_2\to \Fsl_2\times\cdots\times\Fsl_2$ gives rise to an algebra homomorphism
\[\Delta^n : U(\Fsl_2)\to U(\Fsl_2)^{\otimes n}\]
sending $a\in \Fsl_2$ to the sum $\sum_{j=1}^n a^{(j)}$ where $a^{(j)}=1\otimes\cdots\otimes 1\otimes a\otimes 1\otimes\cdots\otimes 1$ where $a$ is in the $j$:th position.

Composing $\Delta_n$ with $(\varphi_1)^{\otimes n}$, and noting that $W(2)^{\otimes n}\cong W(2n)$, the $n$:th Weyl algebra with generators $x^a, \p_a$ ($a=1,2,\ldots,n$) and relations 
\begin{equation}
x^{a}x^{b} - x^{b}x^{a} = \p_a\p_b - \p_b\p_a = \p_ax^b - x^b\p_a - \delta_a^b =0
 \end{equation}
where $\delta_a^b$ denotes the Kronecker delta, we obtain the following known result (see e.g. \cite[7.3.6]{Zhe1994}, \cite{Zhe1996}). 

\begin{Lemma}\label{lem:sl2-to-W2n}
There is a homomorphism of associative unital $\C$-algebras:
\begin{equation}
\varphi_n:U(\Fsl_2) \to W(2n)
\end{equation}
given by
\begin{equation}\label{eq:sl2-to-W2n}
e\mapsto \frac{i}{2}\sum_{a=1}^n\p_a\p_a,\quad f\mapsto \frac{i}{2}\sum_{a=1}^n x^ax^a,\quad h\mapsto -\frac{1}{2}\sum_{a=1}^n(x^a\p_a+\p_a x^a).
\end{equation}
\end{Lemma}

In Section \ref{sec:massless-KG} we will generalize this to arbitrary metric.

\subsection{Reduction Algebras for \texorpdfstring{$\Fsl_2$}{sl(2)}}

Let $\Fg$ be a complex reductive Lie algebra with triangular decomposition $\Fg=\Fn_-\oplus\Fh\oplus\Fn_+$, and $A$ be an associative unital $\C$-algebra with an algebra homomorphism $\varphi:U(\Fg)\to A$. Let $I_+=A\varphi(\Fn_+)$ be the left ideal in $A$ generated by the image of the positive nilpotent part $\Fn_+$ of $\Fg$. Let $N=\{a\in A\mid I_+a\subset I_+\}$ be the normalizer of $I_+$ in $A$, by definition the largest subalgebra of $A$ in which $I_+$ is contained as a two-sided ideal. The \emph{reduction algebra} $Z(A,\Fg)$ associated to $\Fg$ (with its triangular decomposition), $A$, and the map $\varphi$, is defined to be $N/I_+$.

If $V$ is a left $A$-module, it becomes a $U(\Fg)$-module via $\varphi$ and we put $V^+=V^{\Fn_+}=\{v\in V\mid \varphi(a)\cdot v=0\forall a\in\Fn_+\}$. Equivalently, $V^+$ is the set of vectors annihilated by the ideal $I_+$. Consequently $V^+$ carries a natural action of $Z(A,\Fg)$.

It is difficult in general to find elements of $N$ by hand. However, we have the following fact. Let $I_-$ be the right ideal in $A$ generated by $\varphi(\Fn_-)$ and put $\II=I_-+I_+$.

\begin{Proposition}
The map 
\begin{equation}
Q:Z(A,\Fg)\to A/\II,\quad n+I_+\mapsto n+\II
\end{equation}
is injective. Furthermore, if $\varphi(h_\be)+m$ is invertible in $A$ for every coroot $h_\be\in\Fh$ of $\Fg$ and every integer $m$, then $Q$ is surjective. The inverse of $Q$ is given by 
\begin{equation}
a+\II \mapsto P(a+I_+),\quad\forall a\in A
\end{equation}
wher $P=P_\Fg$ is the extremal projector for $\Fg$.
\end{Proposition}

\begin{proof} See, e.g., \cite{KhoOgi2008}.
\end{proof}

\section{Klein--Gordon Reduction Algebra}

\subsection{Oscillator Representation in Generic Metric}
\label{sec:massless-KG}
In this section we generalize the homomorphism from Lemma \ref{lem:sl2-to-W2n}. In \eqref{eq:sl2-to-W2n}, we notice that we are summing over indices on the same level. This signals that we should introduce a metric $\eta_{ab}$. By metric we mean a constant non-degenerate symmetric $n\times n$-matrix. To emphasize that the metric is completely generic, we extend the $n$:th Weyl algebra $W(2n)$ by adjoining a set of $2n^2$ central generators $\eta_{ab}$ and $\eta^{ab}$ subject to relations
\begin{equation}\label{eq:eta-relations}
\eta_{ab}=\eta_{ba},\qquad
\sum_{b=1}^n \eta_{ab}\eta^{bc}=\delta_a^c
\end{equation}
and write $W_\eta(2n)$ for this extended algebra. Here $\delta_a^c$ is the Kronecker delta. There is an isomorphism of associative algebras $W_\eta(2n)=\C[\eta]\otimes W(2n)$, where $\C[\eta]$ is the commutative algebra with generators $\eta_{ab}, \eta^{ab}$, subject to \eqref{eq:eta-relations}. The anti-automorphism $\ast$ we had on $W(2)$ induces a $\C$-linear anti-automorphism of $W_\eta(2n)$ satisfying
\begin{equation}\label{eq:ast-on-W-eta}
(x^a)^\ast = \partial_a,\quad (\partial_a)^\ast=x^a,\quad (\eta_{ab})^\ast=\eta^{ab}
\end{equation}

From now on we employ a \textbf{summation convention} that repeated indices, one subscript and one superscript, are understood to be summed over across the index set $\{1,2,\ldots,n\}$.
We introduce the elements $x_a$ and $\p^a$ of $W_\eta(2n)$ as lowered and raised versions of $x^a$ and $\p_a$, respectively:
\begin{equation}
x_a := \eta_{ab}x^b = \sum_{b=1}^n \eta_{ab}x^b,\qquad 
\p^a := \eta^{ab}\p_b = \sum_{b=1}^n \eta^{ab}\p_b.
\end{equation}
We can recover $x^a$ (respectively $\p_a$) from $x_a$ (respectively $\p^a$) through $x^a=\eta^{ab}x_b$ (respectively $\p_a=\eta_{ab}\p^b$), using \eqref{eq:eta-relations}. Commutators involving these have interesting behavior:
\begin{equation}
[\p^a, x^b] = \eta^{ab}, \qquad [\p_a, x_b] = \eta_{ab},
\end{equation}
where $[z,w]=zw-wz$ for $z,w\in W_\eta(2n)$. (This provides an alternative starting point for defining $W_\eta(2n)$.)

The main point of this section is the following fact:
\begin{Proposition}\label{eq:eta-oscillator}
There is a homomorphism of associative unital $\C$-algebras:
\[
\varphi_{n,\eta}:U(\Fsl_2) \to W_\eta(2n)
\]
given by $e\mapsto E$, $f\mapsto F$, $h\mapsto H$, where
\begin{equation}
E=\frac{i}{2}\p_a\p^a,\quad
F=\frac{i}{2}x_ax^a,\quad
H=-\frac{1}{2}(x^a\p_a+\p_ax^a).
\end{equation}
Here $\p_a\p^a = \sum_{a,b} \eta^{ab}\p_a\p_b$, which becomes the Laplace or d'Alembert operator when $\eta^{ab}$ are specialized to the entries of the identity matrix or diagonal matrix $\diag(1,-1,\ldots,-1)$, respectively.
\end{Proposition}

\begin{proof}
    For completeness, we provide a direct proof of this. We must show that the elements $E,F,H$ satisfy the same relations \eqref{eq:sl2-rels} as $e,f,h$.
    We use $[a,b]=ab-ba$ and $\{a,b\}=ab+ba$. The Leibniz rule implies
    \begin{align*}
        [H,E]&=[-\frac{1}{2}\{x^a, \p_a\}, \frac{i}{2}\p_b\p^b] = -\frac{i}{4}\{[x^a, \p_b\p^b], \p_a\} \\
        &= \frac{i}{4}\{\delta_b^a\p^b + \eta^{ab}\p_b, \p_a\} = \frac{i}{2}\{\p^b, \p_b\} 
        = 2E.
    \end{align*}
    Applying the symplectic Fourier transform \eqref{eq:ast-on-W-eta} to both sides of $-[H,E]=-2E$ we obtain $[H,F]=-2F$.
    Lastly,
    \begin{align*}
        [E,F]&=[\frac{i}{2}\p_a \p^a, \frac{i}{2}x_b x^b] = -\frac{1}{4}(\p_a[\p^a, x_b]x^b + [\p_a, x_b]x^b\p^a + x_b[\p_a, x^b]\p^a + \p_a x_b[\p^a, x^a]) \\
        &= -\frac{1}{4}(\p_a \delta^a_b x^b + \eta_{a, b}x^b\p^a + x_b\delta_a^b\p^a + \p_a x_b \eta^{a, b}) 
        = -\frac{1}{2}\{x^a, \p_a\} = H
    \end{align*}
\end{proof}

The equation $\p_a\p^a \phi = 0$ now has meaning for $\phi$ in the $\eta$-extended polynomial ring \begin{equation}
\mathscr{F}=\C_\eta[x^1,x^2,\ldots,x^n]=\C[\eta]\otimes\C[x^1,x^2,\ldots,x^n].
\end{equation}
This is the natural space that the extended Weyl algebra $W_\eta(2n)$ (and hence $U(\Fsl_2)$, via Proposition \ref{eq:eta-oscillator}) is acting on.

\subsection{Presentation by Generators and Relations}\label{sec:presentation}

To compute relations in the reduction algebra we consider the localized double coset space equipped with the diamond product \cite{KhoOgi2008}. 

In more detail, recall that $E=\varphi_n^\eta(e)$, $F=\varphi_n^\eta(f)$, $H=\varphi_n^\eta(h)$.
Let $A=W'_\eta(2n)=\C(H)\otimes_{\C[H]} W_\eta(2n)$. Let $I_+ = A\cdot E$ be the left ideal in $A$ generated by $E$. Similarly, let $I_-=F\cdot A$ be the right ideal generated by $F$. Let $N=\{n\in A\mid I_+\cdot n\subset I_+\}$ be the normalizer of $I_+$ in $A$ and define the reduction algebra $Z$ to be $N/I_+$.
Let $\II=I_-+I_+$, and consider the double coset space $A/\II$. For $w\in A$, we put $\bar w=w+\II\in A/\II$.
As is known \cite{KhoOgi2008}, the canonical map $Z\to A/\II$, $n+I_+\mapsto n+\II$ is bijective, and the product $\diamond$ on $A/\II$ that makes this an algebra isomorphism is given by
\begin{equation}
    \bar w\diamond \bar z = \overline{wPz}
\end{equation}
where $P$ is the extremal projector~\cite{AshSmiTol1971,asherovaProjectionOperatorsSimple1973,tolstoy2005fortieth,Tol2011}, \cite[3.2.12, 5.1.1]{Zhe1994} for $\Fsl_2$, and $wPz$ should be interpreted using adjoint actions: \begin{equation}
wPz = \sum_{k=0}^\infty \frac{1}{k!}  (\ad F)^k(w)\frac{1}{\psi_k(H)} (\ad E)^k(z),
\end{equation}
and $\psi_0(H)=1$, $\psi_1(H)=H$, and more generally
\begin{equation}
\psi_k(H) = (H+2-2k)(H+3-2k)\cdots (H+1+k-2k).
\end{equation}
(Again, this is a shifted version of the usual expression, because we have moved it to the right of the $F$ power.)
Note that $wPz\in A$ since $\ad E$ acts locally nilpotently on $A$.

We equip $W(2n)$ with the symplectic Fourier transform $\ast$, defined as the unique $\C$-algebra isomorphism $W_\eta(2n)^{\mathrm{op}}\to W_\eta(2n)$ defined by
\begin{equation}
(x^a)^\ast = \p_a,\qquad (\p_a)^\ast=x^a,\qquad (\eta_{ab})^\ast=\eta^{ab}.
\end{equation}
Note that $\ast$ preserves the denominator set, and therefore induces a well-defined operation on the localized algebra $A$. Furthermore, $P^\ast = P$, and $E^\ast = F$, therefore $\II^\ast=\II$. Thus $\ast$ descends to a well-defined involutive $R$-ring anti-automorphism on $A/\II$.

\begin{Lemma}\label{lem:KG-lemma1}
The following identities hold in $A$:
\begin{equation}
[E,x^a] = i\p^a,\qquad [\p_a,F]=ix_a.
\end{equation}
\begin{equation}
x^aH=(H+1)x^a,\qquad H \p_a =\p_a(H+1).
\end{equation}
\end{Lemma}

\begin{proof}
$[E,x^a] = \frac{i}{2}\eta^{bc}[\p_b\p_c,x^a] = \frac{i}{2}\eta^{bc}([\p_b,x^a]\p_c+\p_b[\p_c,x^a])=\frac{i}{2}\eta^{bc}(\delta_b^a\p_c+\p_b\delta_c^a)
=i\p^a$. Then apply the involution $\ast$. The last two follow from that $H$ is an additive constant minus the Euler vector field (or direct calculation).
\end{proof}

\begin{Lemma}\label{lem:KG-lemma2}
The following identities hold in the double coset space $A/\II$:
\begin{align} 
\bar x^a \diamond \bar w &= x^a w+\II,\qquad\forall w\in W(2n), \\
\bar w \diamond \bar \p_a &= w\p_a+\II,\qquad\forall w\in W(2n), 
\end{align}
\end{Lemma}

\begin{proof}
This is immediate by definition of the diamond product because $[F,x^a]=0$ and $[E,\p_a]=0$.
\end{proof}

\begin{Proposition}\label{prop:KG-relations}
The following is a complete presentation of the algebra $A/\II$, hence of the reduction algebra $Z(W'_\eta(2n),\Fsl_2)$:
\begin{subequations}
\begin{align}
\bar x^a \diamond \bar x^b &= \bar x^b \diamond\bar x^a, \\
\bar \p_a \diamond \bar \p_b &= \bar\p_b\diamond\bar\p_a,\\
\bar\p_a\diamond\bar x^b&=\delta_a^b\bar 1+\bar x^b\diamond\bar\p_a+\frac{1}{H+1} \bar x_a\diamond\bar\p^b, \label{eq:dynamical-weyl-rel}\\
\bar x^a H &= (H+1)\bar x^a, \label{eq:xH=(H+1)x}\\
H\bar \p_a &= \bar\p_a (H+1), \label{eq:Hpartial=partial(H+1)}\\
\bar x_a\diamond \bar x^a &= 0, \label{eq:last-relations-1}\\
\bar \p_a \diamond \bar \p^a &=0, \label{eq:last-relations-2}\\
\bar x^a\diamond \bar\p_a &= -(H+\frac{n}{2})\bar 1. \label{eq:last-relations-3}
\end{align}
\end{subequations}
We remind the reader of our summation convention, in effect in the last three relations.
\end{Proposition}

\begin{proof}
The first two relations are immediate by Lemma \ref{lem:KG-lemma2}.
By definition of the diamond product, and Lemma \ref{lem:KG-lemma1}, only two terms survive:
\begin{align*}
\bar\p_aP\bar x^b &=\p_ax^b +[F,\p_a]\frac{1}{H}[E,x^b]+\cdots\II = \\
&= \p_ax^b + (-ix_a)\frac{1}{H}(i\p^b) + \II\\
&=\delta_a^b+x^b\p_a +\frac{1}{H+1}x_a\p^b
\end{align*}
By Lemma \ref{lem:KG-lemma2}, we are done. The next two relations are immediate by Lemma \ref{lem:KG-lemma1}. The orthogonality relations follow from Lemma \ref{lem:KG-lemma2}. The final relation follows from the fact that in $W(2n)$ we have $H=-\frac{1}{2}(x^a\p_a+\p_ax^a)=-\frac{1}{2}(2x^a\p_a + n) = -x^a\p_a -\frac{n}{2}$, hence $x^a\p_a=-(H+\frac{n}{2})$. Now use Lemma \ref{lem:KG-lemma2}.
\end{proof}

Let $R_\eta=\C_\eta(H)$ where $\C_\eta=\C[\eta]$, and $R=\C(H)$.

\begin{Corollary}\label{cor:basis}
Suppose we specialize $\eta_{ab}$ to a diagonal matrix with complex entries. Then a basis for $A/\II$ as a left $R$-module is (here we drop the diamond from the notation):
\begin{equation}
\big\{\bar x_1^{r_1}\cdots \bar x_n^{r_n}\bar\p_1^{s_1}\cdots\bar \p_n^{s_n}\mid r_i,s_j\in \Z_{\ge 0},\, r_n+s_n\le 1\big\}
\end{equation}
\end{Corollary}

\begin{proof}
The algebra homomorphism $U(\Fsl_2)\to W(2n)$ that we consider, factors through $U(\Fsl_2)\otimes W(2n)$ as follows:
\[U(\Fsl_2)\to U(\Fsl_2)\otimes U(\Fsl_2) \to U(\Fsl_2)\otimes W(2n) \to W(2n)\]
The first map is the comultiplication. The second is $1\otimes \varphi$. The last map is $\ep\otimes 1$ where $\ep$ is the counit. Passing to double coset reduction algebras the last map induces an algebra map
\[D(\Fsl_2)\to W'(2n)/\II \]
where $D(\Fsl_2)$ is a differential reduction algebra of $\Fsl_2$. It is of similar type as those considered in \cite{herlemontDifferentialCalculusHDeformed2017,herlemontDifferentialCalculusMathbf2017,herlemontRingsHdeformedDifferential2017}, which were based on the type $A$ oscillator representation. But it has the same kind of basis for the same reason, given as above but without the restrictions on $r_n, s_n$. In other words, it looks like the Weyl algebra. The kernel of the map is the generated by the relations \eqref{eq:last-relations-1}--\eqref{eq:last-relations-3}.
From this the claim follows. 
\end{proof}

\subsection{Explicit Solutions using Raising Operators}\label{sec:explicit-solutions}

\begin{Definition}
Fix non-negative integers $\ell$. Let $a=(a_1,a_2,\ldots,a_\ell)\in\{1,2,\ldots,n\}^\ell$ be a sequence of integers between $1$ and $n$. We introduce the following terminology for certain products of elements in the Weyl algebra $W_\eta(2n)$:
\begin{enumerate}
\item A \emph{mixed $x\p$-monomial indexed by $a$} is a product of the form 
\begin{equation}
z^{a_1}z^{a_2}\cdots z^{a_\ell}
\end{equation}
where for each $1\le i\le\ell$, the symbol $z^{a_i}$ equals $x^{a_i}$ or $\p^{a_i}$.
\item An \emph{ordered $\eta x\p$-monomial indexed by $a$} is a product of the form
\begin{equation}\label{eq:ordered-def}
(\eta^{a_1a_2}\cdots\eta^{a_{2r-1}a_{2r}})(x^{a_{2r+1}}\cdots x^{a_{2r+s}})(\p^{a_{2r+s+1}}\cdots\p^{a_{2r+s+t}})
\end{equation}
where $r,s,t\ge 0$ count the number of $\eta$'s, $x$'s, and $\p$'s respectively.
\end{enumerate}
\end{Definition}

We will use the $\Z$-gradation on $W_{\eta}(2n)$ determined by the condition that $\deg x^a=1$, $\deg \p_a=-1$, $\deg \eta^{ab}=\deg \eta_{ab}=0$, for all $a,b$.

\begin{Lemma}\label{lem:monomial-sums}
Let $\ell$ and $d$ be non-negative integers. Let $a=(a_1,a_2,\ldots,a_\ell)\in\{1,2,\ldots,n\}^\ell$. Then, in the Weyl algebra $W_{\eta}(2n)$, the following two elements are equal:
\begin{enumerate}[{\rm (i)}]
\item the sum of all degree $d$ mixed $x\p$-monomials indexed by $a$;
\item the sum of all distinct degree $d$ ordered $\eta x\p$-monomials indexed by $b$, where $b=(b_1,b_2,\ldots,b_\ell)$ can be obtained from $a$ by permuting the entries.
\end{enumerate}
\end{Lemma}

Before proving this, we consider two examples.

\begin{Example}
When $\ell=3$ and $d=1$, the sum (i) equals
\[\p^{a_1}x^{a_2}x^{a_3}+x^{a_1}\p^{a_2}x^{a_3}+x^{a_1}x^{a_2}\p^{a_3}\]
and the sum (ii) equals
\begin{equation}\label{eq:l3d1-eta-x-partial-monomial-sum}
x^{a_1}x^{a_2}\p^{a_3}+x^{a_1}x^{a_3}\p^{a_2}+x^{a_2}x^{a_3}\p^{a_1} + \eta^{a_1a_2}x^{a_3}+\eta^{a_1a_3}x^{a_2}+\eta^{a_2a_3}x^{a_1}.
\end{equation}
This illustrates that, for the purposes of Lemma \ref{lem:monomial-sums}(ii), $\eta^{a_1a_2}x^{a_3}$ and $\eta^{a_1a_3}x^{a_2}$ are regarded as distinct and both should be included in the sum (ii), regardless of whether or not $a_2=a_3$. On the other hand, only one of the terms $\eta^{a_1a_2}x^{a_3}$ and $\eta^{a_2a_1}x^{a_3}$ should be included; the rationale being that they are always equal due to the symmetry of $\eta$.
\end{Example}

\begin{Example}
When $\ell=4$ and $d=0$ and $a=(1,2,3,4)$, the sum (i) equals
\[\p^1\p^2x^3x^4+\p^1x^2\p^3x^4+\p^1x^2x^3\p^4+x^1\p^2\p^3x^4+x^1\p^2x^3\p^4+x^1x^2\p^3\p^4\]
and the sum (ii) equals
\begin{align*}
&x^1x^2\p^3\p^4+x^1x^3\p^2\p^4+x^1x^4\p^2\p^3+x^2x^3\p^1\p^4+x^2x^4\p^1\p^3+x^3x^4\p^1\p^2\\ 
+&\eta^{12}x^3\p^4 
+\eta^{13}x^2\p^4 
+\eta^{14}x^2\p^3
+\eta^{23}x^1\p^4
+\eta^{24}x^1\p^3
+\eta^{34}x^1\p^2 \\ 
+&\eta^{12}x^4\p^3
+\eta^{13}x^4\p^2
+\eta^{14}x^3\p^2
+\eta^{23}x^4\p^1
+\eta^{24}x^3\p^1
+\eta^{34}x^2\p^1
\end{align*}
In these expression, replacing a superscript $i$ by $a_i$, we obtain the sums (i) and (ii) for general index vector $a=(a_1,a_2,a_3,a_4)$ in the case of $\ell=4$ and $d=0$.
\end{Example}

\begin{proof}[Proof of Lemma \ref{lem:monomial-sums}.]
Each mixed $x\p$-monomial can be rewritten, using the Weyl algebra relation $\p^a x^b=x^b\p^a+\eta^{ab}$, into a sum of ordered $\eta x\p$-monomials. So it suffices to show that each term from the sum (ii) occurs in exactly one term from (i). Consider an ordered $\eta x\p$-monomial indexed by $b$, having $r$ factors of $\eta$'s. We proceed by induction on $r$. If $r=0$, then it is clear that there is only one term it could have come from: the one where the exponents are in the order $a_1,\ldots,a_\ell$. For example the ordered $\eta x\p$-monomial $x^{a_2}x^{a_4}\p^{a_1}\p^{a_4}$ came from the mixed $x\p$-monomial $\p^{a_1}x^{a_2}\p^{a_3}x^{a_4}$. Next, suppose $r>0$. That is, we have $\eta^{a_ia_j}y$ where without loss of generality $i<j$, and $y$ is an ordered $\eta x\p$-monomial with $r-1$ factors of $\eta$ indexed by $a'$, where $a'$ (of length $\ell-2$) is obtained from $a$ by deleting $a_i$ and $a_j$. By the induction hypothesis, $y$ occurs in a unique mixed $x\p$-monomial indexed by $a'$. Inserting $\p^{a_i}$ and $x^{a_j}$ into this mixed monomial so that they appear in the $i$:th and $j$:th slot respectively, gives the required mixed $x\p$-monomial. For example, $\eta^{a_2a_3}x^{a_1}$ occurs when rewriting the mixed $x\p$-monomial $x^{a_1}\p^{a_2}x^{a_3}$ as $x^{a_1}x^{a_3}\p^{a_2}+x^{a_1}\eta^{a_2a_3}$ (as the second term).
\end{proof}

\begin{Lemma}
The following equality holds:
\begin{equation}
P(x^{a_1}+I)\cdots P(x^{a_\ell}+I) = P(x^{a_1}\cdots x^{a_\ell}+I),
\end{equation}
where we regard $P$ as an operator on $W'(2n)/I$.
\end{Lemma}

\begin{proof}
The left hand side, when each factor is expanded as $P(x^a+I)=x^a+\frac{-1}{H+2}F[E,x^a] +I$ and distributing can be written 
\[LHS = x^{a_1}x^{a_2}\cdots x^{a_\ell} + FY_1 + I\]
for some element $Y_1$ of the localized Weyl algebra $W'_\eta(2n)$. Similarly, the RHS directly by definition has the same form:
\[RHS = x^{a_1}x^{a_2}\cdots x^{a_\ell} + FY_2 + I\]
So their difference can be written $F(Y_1-Y_2)+I$. But both sides are elements of the reduction algebra $N'/I'$. Therefore their difference also belongs to $N'/I'$. Since the extremal projector is the identity on this space, we have
\[LHS-RHS = F(Y_1-Y_2) +I= P(F(Y_1-Y_2)+I) = (PF)(Y_1-Y_2+I)=0\]
by the property $PF=0$ of the extremal projector. This proves that LHS=RHS.
\end{proof}

\begin{Remark}
In fact, this proof goes back to Mickelsson's paper~\cite{Mic1973} from 1973 (see also \cite[\S7.3]{Zhe1994}). He realized that elements from the normalizer are uniquely determined by their coset modulo the left ideal generated by the negative nilpotent part of the smaller Lie algebra $\mathfrak{k}$. Another way to phrase this is that the extremal projector, when viewed as a map from $N/I$ to $A/II$ is injective. (In fact it is bijective too, when working in the localized setting.)
\end{Remark}

\begin{Definition}
For any non-negative integer $\ell$ and integers $a_1,a_2,\ldots,a_\ell\in\{1,2,\ldots,n\}$ we define
\begin{equation}
|a_1a_2\cdots a_\ell \rangle = (\tilde x^{a_1}\tilde x^{a_2}\cdots \tilde x^{a_\ell})\cdot 1=P(x^{a_1}x^{a_2}\cdots x^{a_\ell}+I)\cdot 1.
\end{equation}
By construction, these are polynomial solutions to the equation $\partial_i\partial^i\phi=0$.
\end{Definition}

\begin{Definition}
An \emph{ordered $\eta x$-monomial} is an ordered $\eta x\p$-monomial which does not have any $\p$.
\end{Definition}

For example, $\eta^{a_1a_2}x^{a_3}x^{a_4}$ is an ordered $\eta x$-monomial indexed by $(a_1,a_2,a_3,a_4)$.

\begin{Theorem} \label{thm:explicit-solutions}
Suppose $n\ge 3$. We have the following explicit formula:
\begin{equation}\label{eq:explicit-formula}
|a_1a_2\cdots a_\ell\rangle = 
\sum_{k=0}^\infty f_k(-\ell-\frac{n}{2})\cdot(\frac{1}{2}x_bx^b)^k\cdot \sum \begin{pmatrix}\text{all degree $\ell-2k$ ordered}\\\text{$\eta x$-monomials indexed by}\\ \text{some permutation of $a$}\end{pmatrix}
\end{equation}
where $f_k$ is the rational function
\begin{equation}
f_k(t)=\frac{1}{(t+2)(t+3)\cdots(t+1+k)}.
\end{equation}
\end{Theorem}

\begin{proof}
We apply the extremal projector associated to $\Fsl_2$ \cite{AshSmiTol1971,asherovaProjectionOperatorsSimple1973}, \cite[5.1.1]{Zhe1994} to a coset in $A/I$ represented by a product of $x^a$'s.
\begin{equation}
P(x^{a_1}x^{a_2}\cdots x^{a_\ell}+I) = \sum_{k=0}^\infty \frac{(-1)^k}{k!}f_k(H) F^k (\ad E)^k (x^{a_1}x^{a_2}\cdots x^{a_\ell})+I.
\end{equation}
Using that $F=\frac{i}{2}x_b x^b$ and $E=\frac{i}{2}\p_c\p^c$ the imaginary units cancel $(-1)^k$. Furthermore, $\frac{1}{2}[\p_c\p^c , x^a] = \p^a$ by Leibniz rule and $[\p_c,\,x^a]=\delta_c^a$, $[\p^c,\,x^a]=\eta^{ca}$. From $x^a$ to $\p^a$, the degree has dropped by $2$. Applying $\ad\; \frac{1}{2}\p_c\p^c$ again converts another factor of $x$ to a $\p$. Thus, when computing $(\ad \frac{1}{2}\p_c\p^c)^k(x^{a_1}\cdots x^{a_\ell})$, the result is the sum of all mixed $x\p$ monomials with exactly $k$ $\p$'s, or equivalently, of degree $\ell-2k$. Each such monomial is overcounted by a factor of $k!$, which cancels against the $k!$ in the denominator. Thus we get
\begin{equation}
P(x^{a_1}x^{a_2}\cdots x^{a_\ell}+I) = \sum_{k=0}^\infty f_k(H)(\frac{1}{2}x_bx^b)^k \sum \begin{pmatrix}\text{all degree $\ell-2k$ mixed}\\\text{$x\p$-monomials indexed by $a$}\end{pmatrix}+I.
\end{equation}
By Lemma \ref{lem:monomial-sums}, this implies that
\begin{equation}
P(x^{a_1}x^{a_2}\cdots x^{a_\ell}+I) = \sum_{k=0}^\infty f_k(H)(\frac{1}{2}x_bx^b)^k \sum \begin{pmatrix}\text{all degree $\ell-2k$ ordered}\\\text{$\eta x\p$-monomials indexed by}\\ \text{some permutation of $a$}\end{pmatrix}+I.
\end{equation}
Lastly, applying this to the constant solution $1$, any ordered $\eta x\p$-monomials involving a $\p$ will not contribute. The remaining monomials are $\eta x$-monomials. We also use that $H=-\frac{1}{2}(x^a\p_a+\p_ax^a)=-x_a\p^a-\frac{n}{2}$ acts by the scalar $-\ell-\frac{n}{2}$ on any degree $\ell$ polynomial. This proves the claim.
\end{proof}

\begin{Remark}\label{rem:harmonic-projection}
Since every term with $k\ge 1$ in \eqref{eq:explicit-formula} is divisible by $x_bx^b$, the solution $|a_1a_2\cdots a_\ell\rangle$ is the harmonic component of the monomial $x^{a_1}\cdots x^{a_\ell}$ in the decomposition $\mathscr{F}_\ell=\mathscr{F}^+_\ell\oplus x_bx^b\mathscr{F}_{\ell-2}$, where $\mathscr{F}_\ell$ is the space of homogeneous polynomials of degree $\ell$ and $\mathscr{F}^+_\ell=\mathscr{F}_\ell\cap\mathscr{F}^+$. For the Euclidean metric, the fact that this component is obtained by applying raising operators to $1$ is \cite[Prop.~0.5.5]{Zhe1994}, and explicit formulas for the harmonic projection of a homogeneous polynomial are classical; see \cite{AxlRam1995} and \cite[Ch.~5]{AxlBouRam2001}.
\end{Remark}

\begin{Example}
For $\ell=2$ we get, using $f_0=1$, $f_1(-2-\frac{n}{2})=\frac{1}{-2-\frac{n}{2}+2}=-\frac{2}{n}$,
\begin{equation}
|a_1a_2\rangle=P(x^{a_1}x^{a_2}+I)\cdot 1 = x^{a_1}x^{a_2}-\frac{1}{n}x_bx^b \eta^{a_1a_2}.
\end{equation}
It can be checked directly that $\phi=|a_1a_2\rangle$ solves the massless Klein--Gordon equation $\p_c\p^c \phi =0$, for any $a_i\in\{1,2,\ldots,n\}$:
\begin{align*}
\p_c\p^c(|a_1a_2\rangle)&=\p_c(\eta^{ca_1}x^{a_2}+x^{a_1}\eta^{ca_2}-\frac{1}{n}(\delta_b^cx^b+x_b\eta^{cb})\eta^{a_1a_2}) \\ 
&=\eta^{ca_1}\delta_c^{a_2}+\delta_c^{a_1}\eta^{ca_2}-\frac{1}{n}\delta_b^c\delta_c^b\eta^{a_1a_2}-\frac{1}{n}\eta_{cb}\eta^{cb}\eta^{a_1a_2}=0,
\end{align*}
(using $\eta_{ab}\eta^{ab}=\delta_a^a=n$).
\end{Example}

\begin{Example}
Suppose $\ell=3$. Then, in \eqref{eq:explicit-formula}, only the $k=0$ and $k=1$ terms contribute, since $3-2k<0$ for $k\ge 2$ and the degree of an $\eta x$-monomial is always non-negative. For $k=0$, the only degree $3$ $\eta x$-monomial indexed by (a permutation of) $a=(a_1,a_2,a_3)$ is $x^{a_1}x^{a_2}x^{a_3}$. For $k=1$, the sum of all degree $\ell-2k=3-2=1$ $\eta x$-monomials indexed by a permutation of $a$ is (cf. \eqref{eq:l3d1-eta-x-partial-monomial-sum}) $\eta^{a_1a_2}x^{a_3}+\eta^{a_1a_3}x^{a_2}+\eta^{a_2a_3}x^{a_1}$. Thus, for $\ell=3$, \eqref{eq:explicit-formula} states that
\begin{equation}
|a_1a_2a_3\rangle=x^{a_1}x^{a_2}x^{a_3}+\frac{1}{-1-\frac{n}{2}}\cdot\frac{1}{2}x_bx^b \cdot(\eta^{a_1a_2}x^{a_3}+\eta^{a_1a_3}x^{a_2}+\eta^{a_2a_3}x^{a_1}).
\end{equation}
\end{Example}

\begin{Example}
We give the explicit formula for $\ell=4$, in which we get contributions from $k=0,1,2$:
\begin{align}
|a_1a_2a_3a_4\rangle &= x^{a_1}x^{a_2}x^{a_3}x^{a_4} \nonumber \\
&+ \frac{1}{-2-\frac{n}{2}}\cdot\frac{1}{2}x_bx^b\cdot (\eta^{a_1a_2}x^{a_3}x^{a_4}+\eta^{a_1a_3}x^{a_2}x^{a_4}+\eta^{a_1a_4}x^{a_2}x^{a_3} \nonumber \\
&\qquad\qquad\qquad\qquad+\eta^{a_2a_3}x^{a_1}x^{a_4}+\eta^{a_2a_4}x^{a_1}x^{a_3}+\eta^{a_3a_4}x^{a_1}x^{a_2}) \nonumber \\
&+\frac{1}{(-2-\frac{n}{2})(-1-\frac{n}{2})}\cdot (\frac{1}{2}x_bx^b)^2\cdot(\eta^{a_1a_2}\eta^{a_3a_4}+\eta^{a_1a_3}\eta^{a_2a_4}+\eta^{a_1a_4}\eta^{a_2a_3}).
\end{align}
We remind the reader that $x_bx^b=\sum_{a,b=1}^n \eta_{ab}x^a x^b$ is the ``radius squared''. The point of this formula is that for any choice $a_1,a_2,a_3,a_4\in\{1,2,\ldots,n\}$, the degree 4 polynomial $|a_1a_2a_3a_4\rangle$ solves the Klein--Gordon equation. Furthermore, as we will see in the next section, any solution is a linear combination of these solutions. These results are valid in any spacetime dimension $n\ge 3$.
The metric $\eta_{ab}$ can be any flat metric, and $\eta^{ab}$ denote the entries of the inverse of the matrix $(\eta_{ab})$.
\end{Example}

\begin{Remark}
The states $|a_1a_2\cdots a_\ell\rangle$ are not orthogonal. In Section \ref{sec:inner-products} we show that their inner products are related to the matrix elements of symmetrizers from the dynamical R-matrix formalism.
\end{Remark}

\subsection{Irreducibility} \label{sec:irreducibility}

Suppose $n\ge 3$. Let $Z^\circ$ be the $\C$-subalgebra of $Z$ generated by the $2n+1$ generators $\tilde\p_a=\p_a+I_+$, and $\tilde x^a=P(x^a+I_+)=x^a+\frac{1}{H+2}\frac{1}{2}x_bx^b\p^a +I_+$, and $H$.
In this section we show that $Z^\circ$ acts irreducibly on the space $\mathscr{F}^+$ of polynomial solutions to $\p_a\p^a\phi=0$.
In this section we specialize $\eta_{ab}$ to an arbitrary non-degenerate symmetric bilinear form on $\C^n$. 

\begin{Lemma}\label{lem:irreducibility}
Let $B=\C[\p_1,\ldots,\p_n]$. Then any $B$-submodule of $\mathscr{F}^+$ contains the constant solution $1$.
\end{Lemma}

\begin{proof}
The argument is analogous to the proof of non-degeneracy of the pairing $S(V)\otimes S(V^\ast)\to\C$ of constant coefficient differential operators $D\in S(V)$ and polynomials $p\in S(V^\ast)$, $D\otimes p\mapsto \big(D(p)\big)(0)$, see e.g. \cite{chrissRepresentationTheoryComplex2010}. We provide details for convenience.
Let $\phi\in \mathscr{F}^+$ be an arbitrary nonzero element. We wish to show that the cyclic $B$-submodule generated by $\phi$ contains $1$. Since $H$ acts semisimply on $\mathscr{F}^+$, we may without loss of generality assume that $\phi$ is a homogeneous polynomial. Let $d\in\Z_{\ge 0}$ be its degree. Pick one of the nonzero terms in $\phi$, say $\xi (x^1)^{d_1}(x^2)^{d_2}\cdots (x^n)^{d_n}$ where $d_i$ are non-negative integers that sum to $d$, and $\xi$ is a nonzero complex number. Then it is easy to see that 
\[(\p_1)^{d_1}(\p_2)^{d_2}\cdots (\p_n)^{d_n}(\phi) = d_1!d_2!\cdots d_n! \xi\]
which is a nonzero scalar.
Indeed, in any other term in $\phi$ one of the variables must occur in lower degree and is therefore annihilated by the differential operator. Since $\p_a\in B$, this shows that $1$ belongs to the cyclic $B$-submodule generated by $\phi$.
\end{proof}

The following result states a connection between the generators of the localized reduction algebra and solutions to the KG equation.

\begin{Lemma} \label{lem:irreducibility2}
Let $\phi=\sum c_a x^{a_1}\cdots x^{a_\ell}$ be any homogeneous solution to the KG equation of degree $\ell$. Here $c_a\in \C$ and we sum over all sequences $a=(a_1,a_2,\ldots,a_\ell)\in \{1,2,\ldots,n\}^\ell$.
Consider $\tilde\phi$ obtained from $\phi$ by replacing each $x^{a_i}$ by $\tilde x^{a_i}=P(x^{a_i}+I)$. Thus $\tilde\phi$ is an element of the reduction algebra. Assuming $n\ge 3$, we can act on the constant solution $1$. The result is that it reproduces the solution $\phi$:
\begin{equation}
\tilde\phi \cdot 1 = \phi.
\end{equation}
\end{Lemma}

\begin{proof}
The left hand side has the form $\phi + x_bx^b \phi'$ for some degree $\ell-2$ homogeneous polynomial $\phi'$. We claim that $\phi'=0$. To see this, note that since we know the left hand side and $\phi$ are solutions to the KG equation, which is a linear PDE, we also obtain that $x_bx^b\phi'$ solves the KG equation. But the extremal projector is the identity on the solution space. So
\[x_bx^b\phi' = P(x_bx^b \phi') \propto PF \phi'=0.\]
where we used that $x_bx^b$ is proportional to $F$, and that $PF=0$ by definition of the extremal projector for $\Fsl_2$. Since $x_bx^b$ is not a zero-divisor, $\phi'=0$.
\end{proof}

\begin{Theorem}\label{thm:irreducibility}
Let $n$ be a positive integer with $n\ge 3$. Let $\eta_{ab}$ be the entries of a non-degenerate symmetric complex $n\times n$-matrix, and let $\eta^{ab}$ be the entries of its inverse. Let $\mathscr{F}=\C[x^1,x^2,\ldots,x^n]$ the polynomial algebra in $n$ variables over $\C$. Let $\mathscr{F}^+=\{\phi\in \mathscr{F}\mid \sum_{a,b}\eta^{ab}\p_a\p_b\phi=0\}$.
Let $Z^\circ$ be the $\C$-subalgebra of $\operatorname{End}_\C(\mathscr{F})$ generated by the operators
\[H=-\frac{1}{2}\sum_a (x^a\partial_a+\partial_ax^a),\qquad \p_a,\qquad \tilde x^a=x^a+\frac{1}{2H+4}\sum_{b,c}\eta_{bc}x^bx^c\p^a,\quad 1\le a\le n.\]
Then $\mathscr{F}^+$ carries a well-defined action of $Z^\circ$. Moreover, $\mathscr{F}^+$ is irreducible as a module over $Z^\circ$.
\end{Theorem}

\begin{proof}
This is immediate by Lemmas \ref{lem:irreducibility} and \ref{lem:irreducibility2}. We only remark that $2H+4=-2Eu-n+4$ where $Eu$ is the Euler operator. The Euler operator acts on a degree $d$ polynomial with eigenvalue $d$. If $\phi$ has degree $d\ge 1$, then $x_bx^b\p^a\phi$ has degree $d+1$, on which $2H+4$ acts by the scalar $-2(d+1)-n+4=-2d-n+2\le -2d-1<0$, since $n\ge 3$. (If $d=0$, then $\p^a\phi=0$.) Hence $\frac{1}{2H+4}$ is well defined in $\tilde x^a\cdot\phi$.
\end{proof}

\begin{Remark}\label{rem:zhelobenko-z}
For the Laplace operator, Lemma~\ref{lem:irreducibility2} and Theorem~\ref{thm:irreducibility} are due to Zhelobenko \cite[0.5.5, 7.3.6]{Zhe1994}. In his notation, $\tilde x^a$ is the operator $z_a=x_a-x^2\p_a(c+2)^{-1}$ of \cite[0.5.4]{Zhe1994}, where $x^2=\sum_a x_a^2$, $c=2e+n-4$, and $e$ is the Euler operator; indeed $2H+4=-c$.
\end{Remark}

\subsection{Correlation Functions and Dynamical Symmetrizers}\label{sec:inner-products}

The exchange construction \cite{etingof1999exchange} provides a connection between dynamical R-matrices and representation theory of Lie algebras. This connection extends to reduction algebras in the sense that the extremal projector that provides the multiplication in the double coset realization can be identified with the universal solution to the ABRR equation \cite{arnaudon1997universal} for dynamical twist \cite{Kho2004}. This explains the appearance of dynamical R-matrices in \cite{khoroshkinDiagonalReductionAlgebras2010,
herlemontDifferentialCalculusHDeformed2017,
herlemontDifferentialCalculusMathbf2017,
herlemontRingsHdeformedDifferential2017}.

In this section we note (see Theorem \ref{thm:inner-product}) that correlation functions (inner products) between bosonic states in the universal Verma module for the deformed algebra of differential operators acting on Klein-Gordon fields can be expressed as the symmetrizer for the dynamical R-matrix.

Fix $n\ge 3$. We work in the double coset space $\bar A=A/\II$ but drop the bars over generators and write the diamond product as juxtaposition. Thus, the most important relation \eqref{eq:dynamical-weyl-rel} is written
\begin{equation}
    \p_a x^b = \delta_a^b + x^b\p_a+\frac{1}{H+1}x_a\p^b
\end{equation}
Remembering that $x_a=\eta_{ac}x^c$ and $\p^b=\eta^{bd}\p_d$ and introducing Kronecker deltas, we rewrite this equation as follows:
\begin{equation}\label{eq:rel-rewritten}
    \p_a x^b = \delta_a^b + (\delta_c^b\delta_a^d+\eta_{ac}\eta^{bd}\frac{1}{H+1})x^c\p_d
\end{equation}
Using the rational dynamical $R$-matrix (composed with flip)
\begin{equation}
    R_{ac}^{bd}(H)=\delta_c^b\delta_a^d+\eta_{ac}\eta^{bd}\frac{1}{H+1}
\end{equation}
we can write \eqref{eq:rel-rewritten} as
\begin{equation}\label{eq:rel-rewritten-with-R}
    \p_a x^b = \delta_a^b + R_{ac}^{bd} x^c\p_d.
\end{equation}
We consider a generalization of the natural module $\mathscr{F}^+$ which admits a representation of the localized algebra. It can also be viewed as a kind of universal Verma module for the reduction algebra. It is defined as the quotient of the algebra $\bar A$ by the left ideal generated by the $\p_a$'s (remember that we dropped the bars):
\begin{equation}
    \mathcal{M} = \bar A/J,\qquad J=\sum_a \bar A\p_a.
\end{equation}
The space $\mathcal{M}$ can be identified with $\mathbb{C}(H)[x^1,\ldots,x^n]$ modulo the relation $x_ax^a=0$.

There is a natural bilinear form on $\mathcal{M}$ which is analogous to the Shapovalov form on Verma modules for simple Lie algebras (on polynomials, for the Euclidean metric, it is the standard form $\langle f,g\rangle=f(\p)g(0)$ of \cite[0.5.2]{Zhe1994}, where norms of an orthogonal basis of harmonic polynomials are computed \cite[0.5.6--0.5.7]{Zhe1994}):
\begin{equation}
    \langle u\,|\,v\rangle = (u^\ast v)(1)|_{x=0},\qquad u,v\in\mathcal{M}.
\end{equation}
Here we use $(x^a)^\ast = \p_a$ so that $u^\ast$ is a constant-coefficient differential operator. Thus $u^\ast v$ is a differential operator and we act by it on the constant function $1$. Equivalently, $(u^\ast v)(1)=u^\ast(v)$. Lastly, the resulting polynomial is evaluated at the origin by setting all variables $x^a$ to zero. The result is a dynamical scalar, that is, an element of $\mathbb{C}(H)$.
For example,
\[\langle x^a\,|\,x^b\rangle = (\p_a x^b)(1)|_{x=0}=(\delta_a^b+R_{ac}^{bd}(H)x^c\p_d)(1)|_{x=0}=\delta_a^b\]
It is also easy to see that 
\[\langle x^{a_1}\cdots x^{a_r}\,|\,x^{b_1}\cdots x^{b_s}\rangle = 0\qquad\text{if $r\neq s$,}\]
because if $u^\ast v$ does not have degree zero, then $(u^\ast v)(1)|_{x=0}=0$.

To compute the bilinear form on quadratic monomials
\[\langle x^{a_1}x^{a_2}\,|\,x^{b_1}x^{b_2}\rangle  \]
we first compute in the algebra $Z'$:
\begin{align*}
\p_{a_1} x^{b_1}x^{b_2} &= (\delta_{a_1}^{b_1}+R_{a_1 r}^{b_1 s}(H) x^r\p_s)x^{b_2}\\
&=\delta_{a_1}^{b_1}x^{b_2} + R_{a_1r}^{b_1r}(H)x^r(\delta_s^{b_2}+R_{st}^{b_2u}x^t\p_u)\\
&=\delta_{a_1}^{b_1}x^{b_2}+R_{a_1r}^{b_1b_2}(H)x^r + R_{a_1r}^{b_1s}(H)R_{st}^{b_2u}(H+1)x^rx^t\p_u.
\end{align*}
Here we used that $x^r H=(H+1)x^r$. Using this we have
\begin{align*}
\p_{a_2}\p_{a_1}x^{b_1}x^{b_2}&=\p_{a_2}\big(\delta_{a_1}^{b_1}x^{b_2}+R_{a_1r}^{b_1b_2}(H)x^r+R_{a_1r}^{b_1s}(H)R_{st}^{b_2u}(H+1)x^rx^t\p_u\big)\\
&=\delta_{a_1}^{b_1}(\delta_{a_2}^{b_2}+R_{a_2 \al}^{b_2\beta}(H)x^\al\p_\be)
+R_{a_1r}^{b_1b_2}(H-1)(\delta_{a_2}^r+R_{a_2\al}^{r\be}(H)x^\al\p_\be)+\cdots 
\end{align*}
where $\cdots$ means terms that disappear in the bilinear form. Thus we have:
\begin{equation}
\langle x^{a_1}x^{a_2}\,|\,x^{b_1}x^{b_2}\rangle = \delta_{a_1}^{b_1}\delta_{a_2}^{b_2}+ R_{a_1a_2}^{b_1b_2}(H-1).
\end{equation}

For the general formula we need a lemma.

\begin{Lemma}\label{lem:bilinear}
For any positive integer $r$ we have:
\begin{equation}\label{eq:one-partial-many-x}
\big(\p_{a_1}x^{b_1}x^{b_2}\cdots x^{b_r}\big)(1)=S_{a_1c_2\cdots c_r}^{b_1b_2\cdots b_r}(H+r-1) x^{c_2}x^{c_3}\cdots x^{c_r},
\end{equation}
where $S_a^b(H)=\delta_a^b$, and recursively
\begin{equation}\label{eq:S-recursion}
S_{a_1\cdots a_r}^{b_1\cdots b_r}(H+r-1) = 
\delta_{a_1}^{b_1}\delta_{a_2}^{b_2}\cdots\delta_{a_r}^{b_r} + R_{a_1a_2}^{b_1t}(H)S_{ta_3\cdots a_r}^{b_2b_3\cdots b_r}(H+r-1).
\end{equation}
Explicitly, $S_{a_1a_2}^{b_1b_2}(H)=\delta_{a_1}^{b_1}\delta_{a_2}^{b_2}+R_{a_1a_2}^{b_1b_2}(H-1)$ and for general $r$,
\begin{equation}\label{eq:S-definition}
\begin{aligned}
S_{a_1 a_2\cdots a_r}^{b_1b_2\cdots b_r}(H)&=\delta_{a_1}^{b_1}\delta_{a_2}^{b_2}\cdots\delta_{a_r}^{b_r}\\
&\quad+R_{a_1a_2}^{b_1c_2}(H-r+1)\delta_{c_2}^{b_2}\delta_{a_3}^{b_3}\delta_{a_4}^{b_4}\cdots \delta_{a_r}^{b_r} \\
&\quad + R_{a_1a_2}^{b_1c_2}(H-r+1)R_{c_2a_3}^{b_2c_3}(H-r+2) \delta_{c_3}^{b_3}\delta_{a_4}^{b_4}\cdots \delta_{a_r}^{b_r} \\
&\quad +\cdots \\
&\quad+R_{a_1a_2}^{b_1 c_2}(H-r+1)\cdots R_{c_{r-1}a_r}^{b_{r-1}b_r}(H-1).
\end{aligned}
\end{equation}
\end{Lemma}

\begin{proof}
Induction on $r$. For $r=1$ \eqref{eq:one-partial-many-x} is immediate by \eqref{eq:rel-rewritten-with-R}. For $r>1$ we move $x^{b_1}$ to the left using \eqref{eq:rel-rewritten-with-R} in the second term, use that $(x^s D)(1)=x^s \big( D(1)\big)$ for any differential operator $D$, the induction hypothesis, and lastly relation \eqref{eq:xH=(H+1)x}:
\begin{align*}
\big(\p_{a_1}x^{b_1}x^{b_2}\cdots x^{b_r}\big)(1) &=
\big((\delta_{a_1}^{b_1}+R_{a_1 s}^{b_1 t}(H)x^s\p_t)x^{b_2}\cdots x^{b_r}\big)(1) \\
&=\delta_{a_1}^{b_1}x^{b_2}\cdots x^{b_r}+ R_{a_1c_2}^{b_1t}(H)x^{c_2} \cdot \big(\p_t x^{b_2}\cdots x^{b_r}\big)(1)\\
&=\delta_{a_1}^{b_1}x^{b_2}\cdots x^{b_r}+  R_{a_1c_2}^{b_1t}(H)x^{c_2} S_{tc_3\cdots c_r}^{b_2b_3\cdots b_r}(H+r-2) x^{c_3}\cdots x^{c_r}\\
&=\big(
\delta_{a_1}^{b_1}\delta_{c_2}^{b_2}\cdots\delta_{c_r}^{b_r} + R_{a_1c_2}^{b_1t}(H)S_{tc_3\cdots c_r}^{b_2b_3\cdots b_r}(H+r-1) \big) x^{c_2}x^{c_3}\cdots x^{c_r}.
\end{align*}
Thus we are done, by the recursion \eqref{eq:S-recursion}.
\end{proof}

\begin{Theorem}\label{thm:inner-product}
Let $r$ be a non-negative integer. Then
\begin{equation}\label{eq:matrix-elements}
\langle x^{a_1}x^{a_2}\cdots x^{a_r}\,|\,x^{b_1}x^{b_2}\cdots x^{b_r}\rangle = S_{a_1 c_2\cdots c_r}^{b_1b_2\cdots b_r}(H)\cdot S_{a_2d_3\cdots d_r}^{c_2c_3\cdots c_r}(H)\cdot
S_{a_3e_4\cdots e_r}^{d_3d_4\cdots d_r}(H)\cdots S_{a_r}^{z_r}(H)
\end{equation}
where $S$ was given in \eqref{eq:S-definition}, and $z_r$ is the last index being contracted (Einstein summed) over in the product.
\end{Theorem}

\begin{proof}
The formula holds for $r=0$ by convention of empty product being $1$. For $r>0$, it suffices to prove the recursion
\begin{equation}
\langle x^{a_1}x^{a_2}\cdots x^{a_r}\,|\,x^{b_1}x^{b_2}\cdots x^{b_r}\rangle = S_{a_1 c_2\cdots c_r}^{b_1b_2\cdots b_r}(H)\langle x^{a_2}x^{a_3}\cdots x^{a_r}\,|\, x^{c_2}x^{c_3}\cdots x^{c_r}\rangle.
\end{equation}
This follows from Lemma \ref{lem:bilinear} and \eqref{eq:Hpartial=partial(H+1)}.
\end{proof}

For example,
\begin{align*}
\quad&\langle x^{a_1}x^{a_2}x^{a_3}\,|\,x^{b_1}x^{b_2}x^{b_3}\rangle =
S_{a_1 c_2 c_3}^{b_1b_2b_3}(H)\cdot S_{a_2d_3}^{c_2c_3}(H)\cdot
S_{a_3}^{d_3}(H)\\
&=\big(\delta_{a_1}^{b_1}\delta_{c_2}^{b_2}\delta_{c_3}^{b_3}+ R_{a_1c_2}^{b_1r_2}(H-2)\delta_{r_2}^{b_2}\delta_{c_3}^{b_3}+ R_{a_1c_2}^{b_1r_2}(H-2)R_{r_2c_3}^{b_2b_3}(H-1)\big)\\
&\quad\cdot
\big(\delta_{a_2}^{c_2}\delta_{d_3}^{c_3}+R_{a_2 d_3}^{c_2c_3}(H-1)\big)\cdot \delta_{a_3}^{d_3}\\
&=\big(\delta_{a_1}^{b_1}\delta_{c_2}^{b_2}\delta_{c_3}^{b_3}+ R_{a_1c_2}^{b_1b_2}(H-2)\delta_{c_3}^{b_3}+ R_{a_1c_2}^{b_1r_2}(H-2)R_{r_2c_3}^{b_2b_3}(H-1)\big)\\
&\quad\cdot
\big(\delta_{a_2}^{c_2}\delta_{a_3}^{c_3}+R_{a_2 a_3}^{c_2c_3}(H-1)\big)\\
&=\delta_{a_1}^{b_1}\delta_{a_2}^{b_2}\delta_{a_3}^{b_3}+
R_{a_1a_2}^{b_1b_2}(H-1)\delta_{a_3}^{b_3} + 
R_{a_1a_2}^{b_1r_2}(H-2)R_{r_2a_3}^{b_2b_3}(H-1)\\
&+\delta_{a_1}^{b_1}R_{a_2a_3}^{b_2b_3}(H-1)+
R_{a_1c_2}^{b_1b_2}(H-2)R_{a_2a_3}^{c_2b_3}(H-1)+
R_{a_1c_2}^{b_1r_2}(H-2)R_{r_2c_3}^{b_2b_3}(H-1)R_{a_2a_3}^{c_2c_3}(H-1).
\end{align*}

The right hand side of \eqref{eq:matrix-elements} are the matrix elements of the symmetrizer for the dynamical quantum $R$-matrix.
This provides a direct link between the massless Klein--Gordon equation and integrable face models in statistical mechanics. As mentioned, this is not unexpected, since the extremal projector is equivalent to the universal solution $\boldsymbol{J}$ to the ABRR equation for dynamical twists \cite{Kho2004}. But we have not seen this particular aspect of the connection in the literature.

\section*{Author declarations}
\subsection*{Conflict of interest}
The authors have no conflicts to disclose.

\subsection*{Use of AI tools}
In preparing this version, the authors used Claude Opus 5.5 (Anthropic) to compare the paper with Zhelobenko's book \cite{Zhe1994}, to add the corresponding citations and Remarks~\ref{rem:harmonic-projection} and~\ref{rem:zhelobenko-z}, to revise the abstract and the introduction, and to find and correct several errors, including in the statement and proof of Theorem~\ref{thm:irreducibility}, in the examples of Section~\ref{sec:explicit-solutions}, and in the bibliography. The authors have checked all changes and take full responsibility for the content.

\section*{Data availability}
Data sharing is not applicable to this article as no new data were created or analyzed in this study.


\begin{thebibliography}{ABRR98}
\footnotesize\itemsep=0pt
\providecommand{\eprint}[2][]{\href{http://arxiv.org/abs/#2}{arXiv:#2}}

\bibitem[ABRR98]{arnaudon1997universal}
{\sc Arnaudon D., Buffenoir E., Ragoucy E., Roche Ph.,}
{\em Universal solutions of quantum dynamical Yang-Baxter equations},
\href{https://doi.org/10.1023/A:1007498022373}{ Lett. Math. Phys.} \textbf{44} (1998), 201--214. 

\bibitem[AR95]{AxlRam1995}
{\sc Axler S., Ramey W.,}
{\em Harmonic polynomials and Dirichlet-type problems}, Proc. Amer. Math. Soc. \textbf{123} (1995), 3765--3773.

\bibitem[ABR01]{AxlBouRam2001}
{\sc Axler S., Bourdon P., Ramey W.,}
{\em Harmonic function theory}, 2nd ed., Grad. Texts in Math., Vol. 137, Springer, New York, 2001.

\bibitem[AST71]{AshSmiTol1971}
{\sc Asherova R.M., Smirnov Yu.F., Tolstoy V.N.,}
{\em Projection operators for the simple Lie groups},
Theor. Math. Phys. 8, No.2 (1971), 813--825; [Teor. Mat. Fiz. 8, No.2 (1971), 255--271].

\bibitem[AST73]{asherovaProjectionOperatorsSimple1973}
{\sc Asherova R.M., Smirnov Yu.F., Tolstoy, V.N.,}
{\em Projection operators for simple Lie groups: II. General scheme for constructing lowering operators. The groups $\operatorname{SU}(n)$}, \href{http://link.springer.com/10.1007/BF01028268}{Theoret. Math. Phys.} \textbf{15} (1973), 392--401. 

\bibitem[BZ91]{binegarUnitarizationSingularRepresentation1991}
{\sc Binegar B., Zierau R.,}
{\em Unitarization of a singular representation of $\operatorname{SO}(p,q)$},
\href{http://link.springer.com/10.1007/BF02099491}{Comm. Math. Phys.} \textbf{138} (1991) 2, 245--258.

\bibitem[CG10]{chrissRepresentationTheoryComplex2010}
{\sc Chriss N., Ginzburg V.,}
{\em Representation theory and complex geometry},
Modern Birkhäuser Classics, \href{https://link.springer.com/book/10.1007/978-0-8176-4938-8}{Birkhäuser Boston}, Boston, MA, 2010.

\bibitem[DER17]{debieHarmonicTransvectorAlgebra2017}
{\sc De Bie H., Eelbode D., Roels M.,}
{\em The harmonic transvector algebra in two vector variables}, \href{https://linkinghub.elsevier.com/retrieve/pii/S0021869316304215}{J. Algebra} \textbf{473} (2017), 247--282, \eprint{1510.06566}.

\bibitem[D96]{dixmierEnvelopingAlgebras1996}
{\sc Dixmier J.,}
{\em Enveloping algebras}, {Grad. Stud. Math.}, Vol. 11, \href{https://doi.org/10.1090/gsm/011}{American Mathematical Society}, Providence, RI, 1996. Revised reprint of the 1977 translation.

\bibitem[EM22]{eelbodeOrthogonalBranchingProblem2022}
{\sc Eelbode D, Muarem G.,}
{\em The orthogonal branching problem for symplectic monogenics}, \href{https://link.springer.com/10.1007/s00006-022-01215-1}{Adv. Appl. Clifford Algebr.} \textbf{32} (2022), 26 pages, \eprint{2112.00421}.

\bibitem[EM24]{eelbodeHermitianRefinementSymplectic2024}
{\sc Eelbode D, Muarem G.,}
{\em A Hermitian refinement of symplectic Clifford analysis}, \href{https://onlinelibrary.wiley.com/doi/10.1002/mma.10138}{Math. Methods Appl. Sci.} \textbf{47} (2024) 14, 11473--11489, \eprint{2309.08749}.

\bibitem[EV99]{etingof1999exchange}
{\sc Etingof P., Varchenko A.,}
{\em Exchange dynamical quantum groups}, \href{https://doi.org/10.1007/s002200050665}{Comm. Math. Phys.} \textbf{205} (1999), 19--52, \eprint{math/9801135}.

\bibitem[F95]{Fel1995}
{\sc Felder G.,}
{\em Conformal field theory and integrable systems associated to elliptic curves},
Proceedings of the International Congress of Mathematicians Birkhauser, Basel, 1995.

\bibitem[FSS00]{FraSciSor2000}
{\sc Frappat L., Sciarrino A., Sorba P.,}
{\em Dictionary on Lie algebras and superalgebras},
Elsevier, San Diego, 2000.

\bibitem[HW22]{hartwigDiagonalReductionAlgebra2022c}
{\sc Hartwig J.T., Williams\,{II} D.A.,}
{\em Diagonal reduction algebra for $\mathfrak{osp}(1|2)$}, \href{https://link.springer.com/10.1134/S0040577922020015}{ Theoret. Math. Phys.} \textbf{210} (2022), 155--171, \eprint{2106.04380}.

\bibitem[HW23]{hartwigGhostCenterRepresentations2023}
{\sc Hartwig J.T., Williams\,{II} D.A.,}
{\em Ghost center and representations of the diagonal reduction algebra of $\mathfrak{osp}(1|2)$}, \href{https://www.sciencedirect.com/science/article/pii/S0393044023000402}{J. Geom. Phys.} \textbf{187} (2023), 104788, 20 pages, \eprint{2203.08068}.

\bibitem[HW25]{hartwigSymplecticDifferentialReduction2025}
{\sc Hartwig J.T., Williams\,{II} D.A.,}
{\em Symplectic differential reduction algebras and generalized Weyl algebras}, \href{https://www.emis.de/journals/SIGMA/2025/001/}{SIGMA} \textbf{21} (2025), 001, 15 pages, \eprint{2403.15968}.

\bibitem[Ha76]{havlicekCanonicalRealizationsLie1976}
{\sc Havlíček M., Lassner W.,}
{\em Canonical realizations of the Lie algebra $\mathfrak{sp}(2n, \mathbb{R})$}, \href{http://link.springer.com/10.1007/BF01807449}{Internat. J. Theoret. Phys.} \textbf{15} (1976), 867--876. 

\bibitem[HO17a]{herlemontDifferentialCalculusHDeformed2017}
{\sc Herlemont B.,  Ogievetsky O.V.,}
{\em Differential calculus on $\mathbf{h}$-deformed spaces}, \href{https://www.emis.de/journals/SIGMA/2017/082/}{SIGMA} \textbf{13} (2017), 28 pages, \eprint{1704.05330}.
 
\bibitem[H17]{herlemontDifferentialCalculusMathbf2017}
{\sc Herlemont B.,}
{\em Differential calculus on $\mathbf{h}$-deformed spaces}, Ph.D. Thesis, Aix-Marseille Université, 2017, available at \href{https://theses.fr/2017AIXM0377}{https://theses.fr/2017AIXM0377}, \eprint{1802.01357}.

\bibitem[HO17b]{herlemontRingsHdeformedDifferential2017}
{\sc Herlemont B., Ogievetsky O.V.,}
{\em Rings of $\mathbf{h}$-deformed differential operators}, \href{http://link.springer.com/10.1134/S0040577917080104}{Theoret. Math. Phys.} \textbf{192} (2017), 1218--1229, \eprint{1612.08001}.

\bibitem[H85]{How1985}
{\sc Howe R.,}
{\em Dual pairs in physics: Harmonic oscillators, photons, electrons, and singletons},
Lectures in Applied Mathematics, Vol. 21, 1985.

\bibitem[K21]{kalmykovGeometricCategoricalApproaches2021}
{\sc Kalmykov A,}
{\em Geometric and categorical approaches to dynamical representation theory},
Ph.D. Thesis, Universität Zürich, 2021, available at \href{https://github.com/art-kalm/math/blob/master/dyn\_rep\_theory.pdf}{https://github.com/art-kalm/math/blob/master/dyn\_rep\_theory.pdf}.

\bibitem[Kh04]{Kho2004}
{\sc Khoroshkin S.M.,}
{\em Extremal projector and dynamical twist},
Theoretical and mathematical physics 139 (2004) 582--597.

\bibitem[KhO08]{KhoOgi2008}
{\sc Khoroshkin S., Ogievetsky O.,}
{\em Mickelsson algebras and Zhelobenko operators},
Journal of Algebra 319 (2008) 2113--2165.

\bibitem[K39]{kemmerParticleAspectMeson1939}
{\sc Kemmer N.,}
{\em The particle aspect of meson theory}, \href{https://doi.org/10.1098/rspa.1939.0131}{Proc. R. Soc. Lond. A} \textbf{173} (1939), no.~952, 91--116.

\bibitem[M73]{Mic1973}
{\sc Mickelsson J.,}
{\em Step algebras of semi-simple subalgebras of Lie algebras},
Reports on Mathematical Physics 4, no. 4 (1973) 307--318.

\bibitem[T11]{Tol2011}
{\sc Tolstoy V.N.,}
{\em Extremal projectors for contragredient Lie (super) symmetries (short review)},
Physics of Atomic Nuclei 74 (2011) 1747--1757.
\href{https://doi.org/10.48550/arXiv.1010.4054}{arXiv:1010.4054 [math.RT]}

\bibitem[KS22]{kalmykovCategoricalApproachDynamical2022}
{\sc Kalmykov A, Safronov P.,}
{\em A categorical approach to dynamical quantum groups}, \href{https://doi.org/10.1017/fms.2022.68}{Forum Math. Sigma} \textbf{10} (2022), 1--57, \eprint{2008.09081}.
 
\bibitem[KhO10]{khoroshkinDiagonalReductionAlgebras2010}
{\sc Khoroshkin S.M., Ogievetsky O.V.,}
{\em Diagonal reduction algebras of $\mathfrak{gl}$ type}, \href{http://link.springer.com/10.1007/s10688-010-0023-0}{Funct. Anal. Appl.} \textbf{44} (2010), 182--198, \eprint{0912.4055}.

\bibitem[MS25]{mudrovMickelssonAlgebrasHasse2025}
{\sc Mudrov A., Stukopin V.,}
{\em Mickelsson algebras via Hasse diagrams}, \href{https://link.springer.com/10.1007/s10468-024-10311-8}{Algebr. Represent. Theory} \textbf{28} (2025), 353--367.

\bibitem[SD07]{shakeriNumericalSolutionKlein2007}
{\sc Shakeri F., Dehghan M.,}
{\em Numerical solution of the Klein--Gordon equation via He’s variational iteration method}, \href{http://link.springer.com/10.1007/s11071-006-9194-x}{Nonlinear Dyn.} \textbf{51} (2008), 89--97.

\bibitem[To05]{tolstoy2005fortieth}
{\sc Tolstoy V.N.},
{\em Fortieth anniversary of extremal projector method for Lie symmetries}, in: Noncommutative Geometry and Representation Theory in Mathematical Physics, {Contemp. Math.}, Vol. 391, \href{https://mathscinet.ams.org/mathscinet/relay-station?mr=2184036}{American Mathematical Society}, Providence, RI, 2005, 371--384, \eprint{math-ph/0412087}.

\bibitem[V76]{van1976harisch}
{\sc van den Hombergh A.,}
{\em Harish-Chandra modules and representations of step algebra}, Ph.D. Thesis, Catholic University of Nijmegen, 1976, available at \href{http://hdl.handle.net/2066/147527}{http://hdl.handle.net/2066/147527}.

\bibitem[Zh89]{Zhe1989}
{\sc Zhelobenko D.P.,}
{\em Extremal projectors and generalized Mickelsson algebras over reductive Lie algebras},
Mathematics of the USSR-Izvestiya 33, no. 1 (1989) 85--100.

\bibitem[Zh94]{Zhe1994}
{\sc Zhelobenko D.P.,}
{\em Representations of reductive Lie algebras}, Nauka, Moscow, 1994 (in Russian).

\bibitem[Zh96]{Zhe1996}
{\sc Zhelobenko D.P.,}
{\em Hypersymmetries of extremal equations},
Nova J. Theor. Phys 5, no. 4 (1997) 243--258.

\end{thebibliography}
\end{document}